\documentclass[11pt]{article}

\usepackage{amsmath}
\usepackage{latexsym}
\usepackage{amssymb}
\usepackage{graphicx}
\numberwithin{equation}{section}

\newtheorem{theorem}{\bf Theorem}[section]
\newtheorem{definition}{\bf Definition}[section]
\newtheorem{lemma}{\bf Lemma}[section]
\newtheorem{proposition}{\bf Proposition}[section]

\def\udots{\mathinner{\mkern1mu\raise-1pt\vbox{\kern7pt\hbox{.}}\mkern2mu
    \raise2pt\hbox{.}\mkern2mu\raise5pt\hbox{.}\mkern1mu}}

\allowdisplaybreaks

\begin{document}

\begin{center}
{\Large \bf Absolute Continuity of the Laws of the Solutions to Parabolic SPDEs with Two Reflecting Walls}
\end{center}

\begin{center}
Wen Yue
\end{center}

\begin{center}
{\scriptsize Institute of Analysis and Scientific Computing, Technology University of Vienna, Wiedner Haupt. 8-10, Vienna, Austria}
\end{center}

\begin{abstract}
In this paper, we focus on the existence of the density for the law of the solutions to parabolic stochastic partial differential equations with two reflecting walls. The main tool is Malliavin Calculus. 
\end{abstract}


\section{Introduction}

Parabolic SPDEs with reflection are natural extension of the widely studied deterministic parabolic obstacle problems. It was proved by Funaki and Olla in \cite{FO} that the fluctuations of a $\nabla \phi$ interface model near a hard wall converge in law to the stationary solution of an SPDE with reflection. In recent years, there is a growing interest on the study of SPDEs with reflection. Several works are devoted to the existence and uniqueness of the solutions, such as \cite{NP} by Naulart and Pardoux, \cite{XZ} by Xu and Zhang and \cite{O} by Otobe. Especially, the existence and uniqueness of the solution to a fully non-linear SPDE with two reflecting walls was proved by Yang and Zhang \cite{YZ1}. 

We focus here on the existence of the density of the law of the solution, using Malliavin calculus. Malliavin calculus associated with white noise was also used  by Pardoux and Zhang \cite{PZ}, Bally and Pardoux \cite{BP1} to establish the existence of the density of the law of the solution to parabolic SPDE. The case of parabolic stochastic partial differential equation with one reflecting wall was studied by Donati-martin and Pardoux \cite{DP1}. For parabolic SPDEs with two reflecting walls, we construct a convergent sequence $u^{\epsilon,\delta}$ with two indices, based on the case of one reflecting wall. It is more demanding to prove the convergence of $u^{\epsilon,\delta}$ and identify the limit as the solution of the original equation. To prove the positivity of the Malliavin derivative of the solution, we need more delicate partition of sample spaces.

This paper is organized as follows: Section 2 is devoted to fundamental knowledge of parabolic stochastic partial differential equations with two reflecting walls and Malliavin calculus associated with white noise. In Section 3, we recall some results obtained by Yang and Zhang \cite{YZ1} about the existence and uniqueness of the solution to parabolic SPDEs with two reflecting walls and we prove the Malliavin differentiability of the solution. Finally, we give the existence of the density of the law of the solution.

 \section{Preliminaries}

Notation: Let $Q=[0,1]\times R_{+}$, $Q_{T}=[0,1]\times[0,T]$, $V=\{u \in H^{1}([0,1]),u(0)=u(1)=0\}$ where $H^{1}([0,1])$ denotes the usual Sobolev space of absolutely continuous funcitons defined on $[0,1]$ whose derivative belongs to $L^{2}([0,1])$, and $A=-\frac{\partial^{2}}{\partial x^{2}}$.

Consider the following stochastic partial differential equation with two reflecting walls:
\begin{equation}
  \left\{
   \begin{aligned}
   \frac{\partial{u}}{\partial{t}}
   =\frac{\partial^{2}u}{\partial{x}^{2}}+f(x,t,u)+\sigma(x,t,u)\dot{W}(x,t)+\eta-\xi,\\
   u(0,t)=0,u(1,t)=0,\ for\  t\geq0,\\
   u(x,0)=u_{0}(x)\in C([0,1]),\\
   h^{1}(x,t)\leq u(x,t)\leq h^{2}(x,t), for (x,t)\in Q, a.s.\\
   \end{aligned}
  \right.  \label{SPDE two walls}
\end{equation}
where $\dot{W}$ denotes the space-time white noise defined on a complete probability space\\
$(\Omega,\mathcal{F},\{\mathcal{F}_{t}\}_{t\geq 0},P)$, where $\mathcal{F}_{t}=\sigma(W(x,s):x\in [0,1],0\leq s\leq t)$, $u_{0}$ is a continuous function on $[0,1]$, which vanishes at $0$ and $1$.

We assume that the reflecting walls $h^{1}(x,t)$,$h^{2}(x,t)$ are continuous functions satisfying
$h^{1}(0,t),h^{1}(1,t)\leq 0$, $h^{2}(0,t),h^{2}(1,t)\geq 0$, and\\
(H1)$h^{1}(x,t)< h^{2}(x,t)$ for $x\in(0,1)$ and $t\in R_{+}$;\\
(H2)$\frac{\partial{h^{i}}}{\partial{t}}+\frac{\partial^{2}h^{i}}{\partial{x^{2}}}\in L^{2}([0,1]\times [0,T])$, where
$\frac{\partial}{\partial{t}}$ and $\frac{\partial^{2}}{\partial{x^{2}}}$ are interpreted in a distributional sense;\\
(H3)$\frac{\partial}{\partial{t}}h^{i}(0,t)=\frac{\partial}{\partial{t}}h^{i}(1,t)=0$ for $t\geq 0$;\\
(H4)$\frac{\partial}{\partial{t}}(h^{2}-h^{1})\geq 0$.\\
We also assume that the coefficients:
$f,\sigma(x,t,u(x,t)): [0,1]\times R_{+} \times R \rightarrow R$ are measurable and satisfy:\\
$(F):$ $f,\sigma$ are of class of $C^1$ with bounded derivatives with respect to the third element and $\sigma$ is bounded.

The following is the definition of the solution to a parabolic SPDE with two reflecting walls $h^{1}$, $h^{2}$.
\begin{definition}
A triplet $(u,\eta,\xi)$ defined on a filtered probability space \\
$(\Omega, P,\mathcal{F};\{\mathcal{F}_{t}\} )$ is a solution to the SPDE(\ref{SPDE two walls}), denoted by $(u_{0};0,0;f,\sigma;h^{1},h^{2})$, if \\
(i) $u=\{u(x,t);(x,t)\in Q\}$ is a continuous, adapted random field (i.e. $u(x,t)$ is $\mathcal{F}_{t}$-measuralbe $\forall t\geq 0$, $x \in [0,1]$) satisfying $h^{1}(x,t)\leq u(x,t) \leq h^{2}(x,t)$, $u(0,t)=0$ and $u(1,t)=0$, a.s.;\\
(ii) $\eta(dx,dt)$ and $\xi(dx,dt)$ are positive and adapted (i.e. $\eta(B)$ and $\xi(B)$ are $\mathcal{F}_{t}$-measurable if $B\subset(0,1)\times[0,t]$) random measures on $(0,1)\times R_{+}$ satisfying\\
\begin{equation}
\eta((\theta,1-\theta)\times [0,T])< \infty, \xi((\theta,1-\theta)\times[0,T])< \infty\  a.s.
\end{equation}
for $0<\theta<\frac{1}{2}$ and $T>0$;\\
(iii) for all $t\geq 0$ and $\phi \in C_{k}^\infty{}((0,1)\times (0,\infty))$ (the set of smooth functions with compact supports) we have
\begin{eqnarray}\nonumber
(u(t),\phi)-\int_{0}^{t}(u(s),\phi^{''})ds-\int_{0}^{t}(f(y,s,u),\phi)ds-\int_{0}^{t}\int_{0}^{1}\phi \sigma(y,s,u)W(dx,ds)\\\nonumber
=(u_{0},\phi)+\int_{0}^{t}\int_{0}^{1}\phi\eta(dxds)-\int_{0}^{t}\int_{0}^{1}\phi\xi(dxds) a.s.;\\
 \label{u weak form}
\end{eqnarray}
(iv) $\int_{Q}(u(x,t)-h^{1}(x,t))\eta(dx,dt)=\int_{Q}(h^{2}(x,t)-u(x,t))\xi(dx,dt)=0$ a.s..
\end{definition}
Remarks: We note that the stochastic integral in (\ref{u weak form}) is an Ito integral with respect to the Brownian sheet $\{W(x,t); (x,t)\in [0,1]\times R_{+}\}$ defined on the canonical space $\Omega=C_{0}([0,1]\times R_{+})$ (the space of continuous functions on $[0,1]\times R_{+}$ which are zero whenever one of their arguments is zero). The Brownian sheet is equipped with its Borel $\sigma$-field $\mathcal{F}$, the filtration $\mathcal{F}_{t}=\{\sigma(W(x,s)), x\in [0,1], s\leq t\}$ and the Wiener measure $P$.\\
Next, we recall Malliavin calculus associated with white noise:\\
Let $S$ denote the set of "simple random variables" of the form
\begin{eqnarray*}
F=f(W(h_{1}),...,W(h_{n})), n\in N,
\end{eqnarray*}
where $h_{i}\in H:= L^{2}([0,1]\times R_{+})$ and $W(h_{i})$ represent the Wiener integral of $h_{i}$, $f \in C_{p}^{\infty}(R^{n}).$ For such a variable $F$, we define its derivative $DF$, a random variable with values in $L^{2}([0,1]\times R_{+})$ by 
\begin{eqnarray*}
D_{x,t}F=\Sigma_{i=1}^{n}\frac{\partial f}{\partial x_i}(W(h_{1}),...,W(h_{n}))\cdot h_{i}(x,t).
\end{eqnarray*}
We denote by $D^{1,2}$ the closure of $S$ with respect to the norm:
\begin{eqnarray*}
||F||_{1,2}=(E(F^{2}))^{\frac{1}{2}}+[E(||DF||_{L^{2}([0,1]\times R^{+})}^{2})]^{\frac{1}{2}}.
\end{eqnarray*}
$D^{1,2}$ is a Hilbert space. It is the domain of the closure of derivation operator D.\\

We go back to consider the following parabolic SPDE:
\begin{equation}
  \left\{
   \begin{aligned}
   \frac{\partial{u}}{\partial{t}}
   =\frac{\partial^{2}u}{\partial{x}^{2}}+f(x,t,u)+\sigma(x,t,u)\dot{W},\\
   u(0,t)=0,u(1,t)=0,\ for\  t\geq0,\\
   u(x,0)=u_{0}(x)\in C([0,1]),
   \end{aligned}
  \right.  \label{SPDE}
\end{equation}
where $f,\sigma$ satisfy $(F)$.

According to \cite{XZ}, we know $u$ also satisfies the integral equation:
\begin{eqnarray*}
u(x,t)&=&\int_{0}^1G_{t}(x,y)u_{0}(y)dy+\int_{0}^{t}\int_{0}^{1}G_{t-s}(x,y)f(u(y,s))dyds\\
&&+\int_0^t\int_0^1G_{t-s}(x,y)\sigma(u(y,s))W(dyds)
\end{eqnarray*}
And we have the following result from \cite{PZ}.
\begin{proposition}\cite{PZ}
For all $(x,t)\in (0,1)\times R_{+}$, $u(x,t)$ is the solution to (\ref{SPDE}), Then $u(x,t)\in D^{1,2}$ and
$D_{y,s}u(x,t)$ is the solution of SPDE:
\begin{eqnarray*}
D_{y,s}u(x,t)=G_{t-s}(x,y)\sigma(u(y,s))+\int_{s}^{t}\int_{0}^{1}G_{t-r}(x,z)f^{\prime}(u(z,r))D_{y,s}u(z,r)dzdr\\
+\int_s^t \int_0^1 G_{t-r}(x,z)\sigma^{\prime}(u(z,r))D_{y,s}(u(z,r))W(dzdr).
\end{eqnarray*}
\end{proposition}

\section{The Main Result and The Proof} 

We consider the penalized SPDE as follows:
\begin{equation}
\left\{
   \begin{aligned}
   \frac{\partial u^{\epsilon,\delta}(x,t)}{\partial t}-\frac{\partial^{2}u^{\epsilon,\delta}(x,t)}{\partial x^{2}}+f(u^{\epsilon,\delta}(x,t))=\sigma(u^{\epsilon,\delta}(x,t))\dot{W}(x,t)\\
   +\frac{1}{\delta}(u^{\epsilon,\delta}(x,t)-h^{1}(x,t))^{-}-\frac{1}{\epsilon}(u^{\epsilon,\delta}(x,t)-h^{2}(x,t))^{+},\\
   u^{\epsilon,\delta}(0,t)=u^{\epsilon,\delta}(1,t)=0, t\geq 0,\\
   u^{\epsilon,\delta}(x,0)=u_{0}(x), \label{penalized eq}
\end{aligned}
  \right.
\end{equation}
and we can get the following proposition.

\begin{proposition}
If we have (H1),(H2), (H3), (H4) and (F). Then for any $p\geq 1, T>0$, 
$\sup_{\epsilon,\delta}E(||u^{\epsilon,\delta}||_{\infty}^{T})<\infty$ and $u^{\epsilon,\delta}$ converges uniformly on $[0,1]\times[0,T]$ to $u$ as $\epsilon, \delta \to 0$, where $u, u^{\epsilon, \delta}$ are the solutions of SPDE (\ref{SPDE two walls}) and the penalized SPDE (\ref{penalized eq}).
\end{proposition}

\begin{proof}
Let $u^{\epsilon,\delta}$ be the solution to the penalized SPDE (\ref{penalized eq}).\\
Step 1: we prove that there exists $u(x,t)$ such that
\begin{equation}
u:=\lim_{\epsilon \downarrow 0}u^{\epsilon}=\lim_{\epsilon \downarrow 0}\lim_{\delta \downarrow 0}u^{\epsilon,\delta} a.s.
\end{equation}
First fix $\epsilon$,\\
let $v^{\epsilon,\delta}$ be the solution of equation:
\begin{equation}
\left\{
   \begin{aligned}
&&\frac{\partial v^{\epsilon,\delta}(x,t)}{\partial t}-\frac{\partial^{2}v^{\epsilon,\delta}(x,t)}{\partial x^{2}}+f(v^{\epsilon,\delta}(x,t))=\sigma(u^{\epsilon,\delta}(x,t))\dot{W}(x,t)\\
&& \ -\frac{1}{\epsilon}(u^{\epsilon,\delta}(x,t)-h^{2}(x,t))^{+},\\
&&v^{\epsilon,\delta}(x,0)=u_{0}(x), v^{\epsilon,\delta}(0,t)=v^{\epsilon,\delta}(1,t)=0.
\end{aligned}
  \right.
\end{equation}
Then $z^{\epsilon,\delta}=v^{\epsilon,\delta}-u^{\epsilon,\delta}$ is the unique solution in $L^{2}((0,T)\times(0,1))$ of
\begin{equation}
 \left\{
   \begin{aligned}
\frac{\partial z_{t}^{\epsilon,\delta}} {\partial t}+Az_{t}^{\epsilon,\delta}+f(v_{t}^{\epsilon,\delta})-f(u_{t}^{\epsilon,\delta})=-\frac{1}{\delta}(u^{\epsilon,\delta}(x,t)-h^{1}(x,t))^{-},\\
z^{\epsilon,\delta}(x,0)=0,  z^{\epsilon,\delta}(0,t)=z^{\epsilon,\delta}(1,t)=0. \label{z}
\end{aligned}
  \right.
\end{equation}
Multiplying Eq(\ref{z}) by $(z_{s}^{\epsilon,\delta})^{+}$ and integrating it to obtain:
\begin{eqnarray}  \nonumber
&&\int_{0}^{t}(\frac{\partial z^{\epsilon,\delta}(x,s)}{\partial s},(z^{\epsilon,\delta}(x,s))^{+})ds+\int_{0}^{t}(\frac{\partial z^{\epsilon,\delta}(x,s)}{\partial x},\frac{\partial (z^{\epsilon, \delta}(x,s))^{+}}{\partial x})ds\\ \nonumber
&&+\int_{0}^{t}(f(v^{\epsilon,\delta}(x,s))-f(u^{\epsilon,\delta}(x,s)),(z^{\epsilon,\delta}(x,s))^{+})ds\\
&=&-\frac{1}{\delta}\int_{0}^{t}((u^{\epsilon,\delta}(x,s)-h^{1}(x,s))^{-},(z^{\epsilon,\delta}(x,s))^{+})ds.
\end{eqnarray}
According to Bensoussan and Lions \cite{BL} (Lemma 6.1, P132),
$(z_{s}^{\epsilon,\delta})^{+} \in L^{2}(0,T;V)\cap C([0,T];H)$ a.s.
\begin{eqnarray*}
\int_{0}^{t}(\frac{\partial}{\partial s}z_{s}^{\epsilon,\delta},(z_{s}^{\epsilon,\delta})^{+})ds=\frac{1}{2}|(z_{t}^{\epsilon,\delta})^{+}|_{H}^{2}
\end{eqnarray*}
and similarly
\begin{eqnarray*}
\int_{0}^{t}(\frac {\partial} {\partial x}z_{s}^{\epsilon,\delta},\frac{\partial}{\partial x}(z_{s}^{\epsilon,\delta})^{+})ds=\int_{0}^{t}|\frac{\partial}{\partial x}(z_{s}^{\epsilon,\delta})^{+}|^{2}ds\geq 0,
\end{eqnarray*}
and by Lipschitz continuity of $f$, we have
$$\int_{0}^{t}(f(v^{\epsilon,\delta}(x,s))-f(u^{\epsilon,\delta}(x,s)),(z^{\epsilon,\delta}(x,s))^{+})ds
\geq -c\int_{0}^{t}|(z^{\epsilon,\delta}(x,s))^{+}|_{H}^{2}ds,$$
and we deduce that
\begin{eqnarray*}
0 &&\geq \frac{1}{2}|(z^{\epsilon,\delta}(x,t))^{+}|_{H}^{2}+\int_{0}^{t}|\frac{\partial(z^{\epsilon,\delta}(x,s))^{+}}{\partial x}|_{H}^{2}ds-c\int_{0}^{t}|(z^{\epsilon,\delta}(x,s))^{+}|_{H}^{2}ds\\
&&\geq \frac{1}{2}|(z^{\epsilon,\delta}(x,t))^{+}|_{H}^{2}-c\int_{0}^{t}|z^{\epsilon,\delta}(x,s)^{+}|_{H}^{2}ds.
\end{eqnarray*}
Hence,
\begin{equation}
c\int_{0}^{t}|(z^{\epsilon,\delta}(x,s))^{+}|_{H}^{2}ds\geq \frac{1}{2}|z^{\epsilon,\delta}(x,t)^{+}|_{H}^{2}
\end{equation}
From Gronwall's Lemma:
$|(z^{\epsilon,\delta}(x,t))^{+}|_{H}^{2}=0$, $\forall t\geq 0. a.s.$\\
Then, 
\begin{eqnarray}
u^{\epsilon,\delta}(x,t)\geq v^{\epsilon,\delta}(x,t), \forall x\in [0,1], t\geq 0. a.s.  \label{u's lower bound}
\end{eqnarray}
From Theorem 3.1 in $\cite{DP}$, we get that
the following equation has a unique solution $\{w^{\epsilon,\delta}(x,t);x\in[0,1],t\geq 0\}$:
\begin{equation*}
\left\{
   \begin{aligned}
   &&\frac{\partial w^{\epsilon,\delta}(x,t)} {\partial t}-\frac{\partial^{2}w^{\epsilon,\delta}(x,t)}{\partial x^{2}}+f(w^{\epsilon,\delta}(x,t)+\sup_{s\leq t, y\in[0,1]}(w^{\epsilon,\delta}(y,s)-h^{1}(y,s))^{-})\\
&&=\sigma(u^{\epsilon,\delta}(x,t)) \dot{W}(x,t)-\frac{1}{\epsilon}(u^{\epsilon,\delta}(x,t)-h^{2}(x,t))^{+},\\
&&w^{\epsilon,\delta}(\cdot , 0)=u_{0}, w^{\epsilon,\delta}(0,t)=w^{\epsilon,\delta}(1,t)=0.
\end{aligned}
  \right.
\end{equation*}

We set
\begin{equation}
\overline{w}^{\epsilon,\delta}(x,t)=w^{\epsilon,\delta}(x,t)+\sup_{s\geq t,y\in[0,1]}(w^{\epsilon,\delta}(y,s)-h^{1}(y,s))^{-}
=w^{\epsilon,\delta}(x,t)+\Phi_{t}^{\epsilon,\delta}
\end{equation}
$\overline{w}^{\epsilon,\delta}(x,t)-h^{1}(x,t)\geq 0$ and $\Phi_{t}^{\epsilon,\delta}$ is an increasing process.\\
For any $T>0$, $\overline{z}^{\epsilon,\delta}=u^{\epsilon,\delta}-\overline{w}^{\epsilon,\delta}$ is the unique solution in $L^{2}((0,T);H^{1}(0,1))$ of
\begin{equation*}
\left\{
   \begin{aligned}
   &&\frac{\partial \overline{z}^{\epsilon,\delta}(x,t)} {\partial t}+A\overline{z}^{\epsilon,\delta}(x,t)+f(u^{\epsilon,\delta}(x,t))-f(\overline{w}^{\epsilon,\delta}(x,t))+\frac{d\Phi_{t}^{\epsilon,\delta}} {dt}\\
   &&=\frac{1}{\delta}(u^{\epsilon,\delta}(x,t)-h^{1}(x,t))^{-},\\
&& \overline{z}^{\epsilon,\delta}(\cdot ,0)=0,\\
&& \overline{z}^{\epsilon,\delta}(0,t)=\overline{z}^{\epsilon,\delta}(1,t)=-\Phi_{t}^{\epsilon,\delta}.
\end{aligned}
  \right.
\end{equation*}
Multiplying this equation by $(\overline{z}^{\epsilon,\delta}(x,s))^{+}$, we obtain by the same arguments as above:
\begin{eqnarray}\nonumber
&&\int_{0}^{t}(\frac{\partial \overline{z}^{\epsilon,\delta}(x,s)} {\partial s},(\overline{z}^{\epsilon,\delta}(x,s))^{+})ds
+\int_{0}^{t}(\frac{\partial\overline{z}^{\epsilon,\delta}(x,s)}{\partial x},\frac{\partial(\overline{z}^{\epsilon,\delta}(x,s))^{+}}{\partial x})ds\\ \nonumber
&& +\int_{0}^{t}(f(u^{\epsilon,\delta}(x,s))-f(\overline{w}^{\epsilon,\delta}(x,s)),(\overline{z}^{\epsilon,\delta}(x,s))^{+})ds\\\nonumber
&&+\int_{0}^{t}\int_{0}^{1}(\overline{z}^{\epsilon,\delta}(x,s))^{+}dxd\Phi_{s}^{\epsilon,\delta}\\
&&=\frac{1}{\delta}\int_{0}^{t}((u^{\epsilon,\delta}(x,s)-h^{1}(x,s))^{-},(\overline{z}^{\epsilon,\delta}(x,s))^{+})ds
\end{eqnarray}
The right-hand side of the above equality is zero because $(\overline{z}^{\epsilon,\delta}(x,s))^{+}>0$ implies $u^{\epsilon,\delta}(x,s)-h^{1}(x,s)>\overline{w}^{\epsilon,\delta}(x,s)-h^{1}(x,s)\geq 0$.\\
Hence we again deduce from Gronwall's Lemma:
\begin{eqnarray}
u^{\epsilon,\delta}(x,t)\leq \overline{w}^{\epsilon,\delta}(x,t) \label{u's upper bound}
\end{eqnarray}
By (\ref{u's lower bound}),(\ref{u's upper bound}),
\begin{eqnarray}\nonumber
|u^{\epsilon,\delta}(x,t)| &\leq& |v^{\epsilon,\delta}(x,t)|+|w^{\epsilon,\delta}(x,t)|+\sup_{s\leq t,y\in[0,1]}(w^{\epsilon,\delta}(y,s)-h^{1}(y,s))^{-}\\
&&\leq |v^{\epsilon,\delta}(x,t)|+2\sup_{s\leq t,y\in[0,t]}[|w^{\epsilon,\delta}(y,s)|+|h^{1}(y,s)|].
\end{eqnarray}
From Lemma 6.1 in \cite{DP}, for arbitrarily large $p$ and any $T>0$, consider that $f^{\prime} (v^{\epsilon,\delta}(x,t))=f(v^{\epsilon,\delta}(x,t))+\frac{1}{\epsilon}(u^{\epsilon,\delta}(x,t)-h^2(x,t))^{+}$ is Lipschitz continuous with respect to 
$v^{\epsilon, \delta}$ and
$f^{\prime}(w^{\epsilon,\delta}(x,t)+\sup_{s \geq t, y\in [0,1]}(w^{\epsilon,\delta}(y,s)-h^1(y,s))^{-})=\\
f(w^{\epsilon,\delta}(x,t)+\sup_{s \geq t, y\in [0,1]}(w^{\epsilon,\delta}(y,s)-h^1(y,s))^{-})+\frac{1}{\epsilon}(u^{\epsilon,\delta}(x,t)-h^2(x,t))^{+}$
is Lipschitz continuous with respect to $w^{\epsilon,\delta}(x,t)+\sup_{s \geq t, y\in [0,1]}(w^{\epsilon,\delta}(y,s)-h^1(y,s))^{-}$,
we have that\\
$\sup_{\delta}E[\sup_{(x,t)\in \overline{Q}_{T}}|v^{\epsilon,\delta}(x,t)|^{p}]<\infty$
and
$\sup_{\delta}E[\sup_{(x,t)\in \overline{Q}_{T}}|w^{\epsilon,\delta}(x,t)|^{p}]<\infty,$\\
which imply
\begin{eqnarray}
\sup_{\delta}E[\sup_{(x,t)\in \overline{Q}_{T}}|u^{\epsilon,\delta}(x,t)|^{p}]<\infty.
\end{eqnarray}
So $u^{\epsilon}=\sup_{\delta}u^{\epsilon,\delta}$ is a.s. bounded on $\overline{Q}_{T}$. \\
Let
\begin{equation}
\eta^{\epsilon}=\lim_{\delta \rightarrow 0}\frac{(u^{\epsilon,\delta}(x,t)-h^{1}(x,t))^{-}}{\delta}
\end{equation}
Similar as the proof of Th4.1 in \cite{DP}, $u^{\epsilon}$ is continuous and $u^{\epsilon}$ is the solution to:
\begin{equation}
\frac{\partial u^{\epsilon}} {\partial t}+Au^{\epsilon}+f(u^{\epsilon})=\sigma(u^{\epsilon})\dot{W}(x,t)+\eta^{\epsilon}(x,t)-\frac{1}{\epsilon}(u^{\epsilon}(x,t)-h^{2}(x,t))^{+}
\end{equation}
In addition, by the definition of $u^{\epsilon}$, $u^{\epsilon}\geq h^{1}$and using Theorem 1.2.6 (Comparison Theorem)$, u^{\epsilon}$ decreases when $\epsilon \rightarrow 0$.\\
Hence, there exists $u(x,t)$ such that
\begin{equation}
u:=\lim_{\epsilon \downarrow 0}u^{\epsilon}=\lim_{\epsilon \downarrow 0}\lim_{\delta \downarrow 0}u^{\epsilon,\delta} a.s.
\end{equation}
Step 2: Next we prove $u(x,t)$ is continuous.\\
Let $\tilde{v}^{\epsilon,\delta}$ be the solution of
\begin{equation}
\frac{\partial \tilde{v}^{\epsilon,\delta}}{\partial t}+A\tilde{v}^{\epsilon,\delta}=\sigma(u^{\epsilon,\delta})\dot{W},
\end{equation}
and let $\hat{v}$ be the solution of
\begin{equation}
\frac{\partial \hat{v}}{\partial t}+A \hat{v}=\sigma(u)\dot{W}.
\end{equation}
Remember
\begin{eqnarray*}
   \frac{\partial u^{\epsilon,\delta}(x,t)}{\partial t}-\frac{\partial^{2}u^{\epsilon,\delta}(x,t)}{\partial x^{2}}+f(u^{\epsilon,\delta}(x,t))=\sigma(u^{\epsilon,\delta}(x,t))\dot{W}(x,t)\\
   +\frac{1}{\delta}(u^{\epsilon,\delta}(x,t)-h^{1}(x,t))^{-}-\frac{1}{\epsilon}(u^{\epsilon,\delta}(x,t)-h^{2}(x,t))^{+},
\end{eqnarray*}
Let $\tilde z^{\epsilon,\delta}=u^{\epsilon,\delta}-\tilde{v}^{\epsilon,\delta}$,
then $\tilde z^{\epsilon,\delta}$ is the solution of
\begin{eqnarray}\nonumber
&&\frac{\partial \tilde z^{\epsilon,\delta}}{\partial t}+A\tilde z^{\epsilon,\delta}+f(\tilde z^{\epsilon,\delta}+\tilde{v}^{\epsilon,\delta})\\
&=&\frac{1}{\delta}(\tilde z^{\epsilon,\delta}+\tilde{v}^{\epsilon,\delta}-h^{1})^{-}-\frac{1}{\epsilon}(\tilde z^{\epsilon,\delta}+\tilde{v}^{\epsilon,\delta}-h^{2})^{+}.
\end{eqnarray}
Let $\hat {z}^{\epsilon,\delta}$ be the solution of
\begin{eqnarray}
\frac{\partial \hat{z}^{\epsilon,\delta}}{\partial t}+A\hat{z}^{\epsilon,\delta}+f(\hat{z}^{\epsilon,\delta}+\hat{v})=\frac{1}{\delta}(\hat{z}^{\epsilon,\delta}+\hat{v}-h^{1})^{-}-\frac{1}{\epsilon}(\hat{z}^{\epsilon,\delta}+\hat{v}-h^{2})^{+}.
\end{eqnarray}
We have 
\begin{eqnarray}
||\tilde{z}^{\epsilon,\delta}-\hat z^{\epsilon,\delta}||_{T,\infty}\leq ||\tilde{v}^{\epsilon,\delta}-\hat{v}||_{T,\infty}. \label{tilde and hat}
\end{eqnarray}
$\hat{z}^{\epsilon,\delta}$ is continuous.
According to proof of Theorem 2.1 in \cite{O} , $\hat{z}^{\epsilon,\delta}\rightarrow \hat{z}\ (continuous)$. \\
It means
$$\hat{z}=\lim_{\epsilon\rightarrow 0}\hat{z}^{\epsilon}=\lim_{\epsilon\rightarrow 0}\lim_{\delta \rightarrow 0}\hat{z}^{\epsilon,\delta}.$$
Fix $\epsilon$, $\hat{z}^{\epsilon,\delta} \uparrow \hat{z}^{\epsilon}(continuous)$, and from Dini theorem,
$\hat{z}^{\epsilon,\delta}$ uniformly converges to $\hat{z}^{\epsilon}$.
i.e.$||\hat{z}^{\epsilon,\delta}-\hat{z}^{\epsilon}||_{T,\infty}\rightarrow 0,\ \ \delta \rightarrow 0$.\\
Since $\hat{z}^{\epsilon}\downarrow \hat{z}$, and from Dini theorem, $\hat{z}^{\epsilon}$ uniformly converges to $\hat{z}$. 
i.e.$||\hat{z}^{\epsilon}-\hat{z}||_{T,\infty} \rightarrow 0$.\\
Then we get
\begin{eqnarray}\nonumber
&&||\hat{z}^{\epsilon,\delta}-\hat{z}||_{T,\infty}=||\hat{z}^{\epsilon,\delta}-\hat{z}^{\epsilon}+\hat{z}^{\epsilon}-\hat{z}||_{T,\infty}\leq ||\hat{z}^{\epsilon,\delta}-\hat{z}^{\epsilon}||_{T,\infty}+||\hat{z}^{\epsilon}-\hat{z}||_{T,\infty}\rightarrow 0\\
 &&(\delta \rightarrow 0, \epsilon \rightarrow 0).
\label{hat and hat}
\end{eqnarray}
i.e. $\hat{z}^{\epsilon,\delta} \rightarrow \hat{z}$ uniformly.\\
Next we prove $\tilde{v}^{\epsilon,\delta}\rightarrow \hat v$ uniformly with respect to $s,t$ as $\epsilon\rightarrow 0, \delta \rightarrow 0$:\\
Let $I(x,t)=\tilde{v}^{\epsilon,\delta}(x,t)-\hat{v}(x,t)=\int_0^t\int_0^1G_{t-s}(x,y)(\sigma(u^{\epsilon,\delta})-\sigma(u))W(dyds)$,
from the proof of Corollary 3.4 in \cite{W2},
$$E|I(x,t)-I(y,s)|)^p 
\leq
C_T E\int_0^{t\bigvee s} \int_{0}^{1}(|\sigma(u^{\epsilon,\delta})-\sigma(u)|)^p dzdr |(x,t)-(y,s)|^{\frac{p}{4}-3},
$$
and following the same calculation as in the proof of Theorem 2.1 in Xu and Zhang \cite{XZ}, we deduce
\begin{eqnarray*}
E(\sup_{x\in [0,1],t\in[0,T]}|I(x,t)|)^p 
\leq
C_T E\int_0^T\int_{0}^{1}(|\sigma(u^{\epsilon,\delta})-\sigma(u)|)^p dxdt.
\end{eqnarray*}
Again according to $u:=\lim_{\epsilon\rightarrow 0}\lim_{\delta\rightarrow 0}u^{\epsilon,\delta}$ and $\sigma(x,t,u(x,t))$ is Lipschitz continuous and bounded,
we can have 
\begin{eqnarray*}
E(\sup_{x\in [0,1],t\in [0,T]}|I(x,t)|)^p 
&& \leq
C_T E\int_0^T\int_{0}^{1}(|\sigma(u^{\epsilon,\delta})-\sigma(u)|)^p dtdx\\
&&\rightarrow 0
\end{eqnarray*}
Then we have that $\tilde{v}^{\epsilon,\delta}\rightarrow \hat v$ uniformly a.s. and again from (\ref{tilde and hat}) and (\ref{hat and hat})  we deduce that  $\tilde z^{\epsilon,\delta}\rightarrow \hat{z}$ uniformly a.s..\\
So $$\lim_{\epsilon\rightarrow 0}\lim_{\delta \rightarrow 0}u^{\epsilon,\delta}=u=\hat{z}+\hat{v}$$ is continuous.\\
Step 3: Next we prove $u(x,t)$ is the solution of
\begin{equation}
\frac{\partial u}{\partial t}+Au+f(u)=\sigma(u)\dot{W}(x,t)+\eta(x,t)-\xi(x,t).
\end{equation}
For $\psi \in C_{0}^{\infty}((0,1)\times [0,\infty))$,
\begin{eqnarray}\nonumber
&&-\int_{0}^{t}(u^{\epsilon}(x,s),\psi_{s}(s))ds-\int_{0}^{t}(u^{\epsilon}(x,s),A \psi)ds+\int_{0}^{t}(f(u^{\epsilon}),\psi)ds\\ \nonumber
&=&\int_{0}^{t}\int_{0}^{1}(\sigma(u^{\epsilon}),\psi)W(dx,ds)+\int_{0}^{t}\int_{0}^{1}\psi(x,t)(\eta^{\epsilon}(dx,dt)-\xi^{\epsilon}(dx,dt)) \\
\label{u sequence}
\end{eqnarray}
$$\eta^{\epsilon}=\lim_{\delta \rightarrow 0}\frac{(u^{\epsilon,\delta}-h^{1})^{-}}{\delta}, \xi^{\epsilon}=\frac{(u^{\epsilon}-h^{2})^{+}}{\epsilon}.$$
Let $\epsilon \rightarrow 0,$
\begin{eqnarray*}
&&-\int_{0}^{t}(u(x,s),\psi_{s})ds-\int_{0}^{t}(u(x,s),A \psi)ds+\int_{0}^{t}(f(u),\psi)ds\\
&=&\int_{0}^{t}\int_{0}^{1}(\sigma(u),\psi)W(dx,ds)+\lim_{\epsilon \rightarrow 0}\int_{0}^{t}\int_{0}^{1}\psi(x,t)(\eta^{\epsilon}(dx,dt)-\xi^{\epsilon}(dx,dt)).
\end{eqnarray*}
Then it is clear that, under the limit $\epsilon\rightarrow 0$,
$lim_{\epsilon\rightarrow 0}(\eta^{\epsilon}-\xi^{\epsilon})$ exists in the sense of Schwartz distribution a.s..\\
Because $u^{\epsilon}$ uniformly converges to $u$, similarly as Theorem 3.1 in \cite{YZ1} we get $\eta^{\epsilon} \rightarrow \eta$ and $\xi^{\epsilon}\rightarrow \xi$.
Let $\epsilon \rightarrow 0$ to see that $(u,\eta,\xi)$ satisfies condition (iii) of Def 3.2.1.\\
Multiplying both sides of Eq(\ref{u sequence}) by $\epsilon$ and letting $\epsilon\rightarrow 0$,
\begin{eqnarray}
\lim_{\epsilon\rightarrow 0} \int_{0}^{t}\int_{0}^{1}\psi(x,t)(\epsilon \lim_{\delta\rightarrow 0}\frac{(u^{\epsilon,\delta}-h^{1})^{-}}{\delta}-(u^{\epsilon}-h^{2})^{+})(dx,dt)=0
\end{eqnarray}
then $\int_{0}^{t}\int_{0}^{1}\psi(x,t)(u-h^{2})^{+}(dx,dt)=0$, and we can get $u\leq h^{2}$. And since $u^{\epsilon}\geq h^{1}$,
then $u\geq h^{1}$. Combining these two inequalities, we have $h^{1}\leq u\leq h^{2}$.\\
Finally, we can show that $\int_{Q_{T}}(u-h^{1})d\eta=\int_{Q_{T}}(h^{2}-u)d\xi=0$.\\
For $\epsilon \leq \epsilon^{'}, \ u^{\epsilon}\geq u^{\epsilon^{'}}$,
therefore $supp(\eta^{\epsilon})\subset supp(\eta^{\epsilon^{'}})$,
we get $supp(\eta)\subset supp(\eta^{\epsilon})$.
we know $u^{\epsilon}-h^{1}\leq 0$ on $supp \eta^{\epsilon}$.
So $\int_{Q_{T}}(u^{\epsilon}-h^{1})d\eta\leq 0$.
Then$\int_{Q_{T}}(u-h^{1})d\eta= 0$.
Because $\xi^{\epsilon}=\frac{1}{\epsilon}(u^{\epsilon}-h^{2})^{+},$ then $0\geq \int_{Q_{T}}(u^{\epsilon}-h^{2})d\xi^{\epsilon}\geq 0$. And since $\xi^{\epsilon} \rightarrow \xi$, then $\int_{Q_{T}} (u-h^{2})d\xi=0$.\\
By taking $\psi \in C_{0}^{\infty}((0,1)\times (0,\infty))$ such that $\psi =1$ on $(supp \eta) \cap ((\delta, 1-\delta)\times [0,T])$ and $\psi=0$ on $supp \xi$. Hence, in view of (\ref{u weak form}),
 $$\eta([\delta,1-\delta]\times [0,T])= \int_{0}^{T}\int_{0}^{1}\psi(x,t)\eta(dx,dt)-\int_0^T \phi(x,t)\xi(dx,dt)<\infty$$ for all $0<\delta<\frac{1}{2}$ and $T>0$.
Similarly  we can get $\xi([\delta,1-\delta]\times[0,T])<\infty$ for all $0<\delta<\frac{1}{2}$ and $T>0$. $\Box$
\end{proof}
\\

Set $k_1(u^{\epsilon,\delta}-h^1(x,t))=arctan[(u^{\epsilon,\delta}-h^1(x,t))\wedge 0]^{2}$ and $k_{2}(u^{\epsilon,\delta}-h^2(x,t))=arctan[(h^2(x,t)-u^{\epsilon,\delta})\wedge 0]^{2}$. Consider the following penalized SPDE:
\begin{equation}
\left\{
   \begin{aligned}
   \frac{\partial u^{\epsilon,\delta}(x,t)}{\partial t}-\frac{\partial^{2}u^{\epsilon,\delta}(x,t)}{\partial x^{2}}+f(u^{\epsilon,\delta}(x,t))=\sigma(u^{\epsilon,\delta}(x,t))\dot{W}(x,t)\\
   +\frac{1}{\delta}k_1(u^{\epsilon,\delta}-h^1(x,t))-\frac{1}{\epsilon}k_{2}(u^{\epsilon,\delta}-h^2(x,t)),\\
   u^{\epsilon,\delta}(x,0)=u_{0}(x).
\end{aligned}
  \right.
\end{equation}
Notice that the corresponding penalized elements in Proposition 3.3.1 are $(u^{\epsilon,\delta}-h^{1}(x,t))^-$ and$(u^{\epsilon,\delta}-h^2(x,t))^+$.
It was shown in \cite{DMZ}(also in \cite{DP1}) that the choice of $k_1,k_2$ does not change the limit of $u^{\epsilon,\delta}$, but makes $k_{1},k_{2}$ differentiable with respect to $u^{\epsilon,\delta}$.

\begin{proposition}
For all $(x,t) \in [0,1]\times R^{+}$, $u(x,t) \in D_{1,p}$ and there exists a subsequence of $Du^{\epsilon,\delta}(x,t)$ that converges to $Du(x,t)$ in the weak topology of $L^{p}(\Omega;H)$ and $H=L^{2}([0,1]\times R^{+})$.
\end{proposition}

\begin{proof}
Let $u^{\epsilon,\delta}$ be the solution to the following SPDE:
\begin{equation}
\left\{
   \begin{aligned}
   \frac{\partial u^{\epsilon,\delta}(x,t)}{\partial t}-\frac{\partial^{2}u^{\epsilon,\delta}(x,t)}{\partial x^{2}}+f(u^{\epsilon,\delta}(x,t))=\sigma(u^{\epsilon,\delta}(x,t))\dot{W}(x,t)\\
   +\frac{1}{\delta}k_1(u^{\epsilon,\delta}-h^1(x,t))-\frac{1}{\epsilon}k_{2}(u^{\epsilon,\delta}-h^2(x,t)),\\
   u^{\epsilon,\delta}(x,0)=u_{0}(x).
\end{aligned}
  \right.
\end{equation}
Then it can be expressed as,
\begin{eqnarray*}
u^{\epsilon,\delta}(x,t)
&=&\int_{0}^{t}G_{t}(x,y)u_{0}(y)dy+\int_{0}^{t}\int_{0}^{1}G_{t-s}(x,y)\sigma(u^{\epsilon,\delta}(x,t))W(dyds)\\
&&+\int_{0}^{t}\int_{0}^{1}G_{t-s}(x,y)[-f(u^{\epsilon,\delta}(x,t))+\frac{1}{\delta}k_{1}-\frac{1}{\epsilon}k_{2}]dyds,
\end{eqnarray*}
where $G_{t}(x,y)$ is the heat kernel.\\
And we also know from Section 3.2 that:
\begin{eqnarray*}
D_{y,s}u^{\epsilon,\delta}(x,t)
&=&G_{t-s}(x,y)\sigma(u^{\epsilon,\delta}(y,s))\\
&&+\int_{s}^{t}\int_{0}^{1}G_{t-r}(x,z)\sigma^{'}(u^{\epsilon,\delta}(z,r))D_{y,s}(u^{\epsilon,\delta}(z,r))W(dzdr) \\
&&+\int_{s}^{t}\int_{0}^{1}G_{t-r}(x,z)[-f^{'}+\frac{1}{\delta}k_{1}^{'}-\frac{1}{\epsilon}k_{2}^{'}]D_{y,s}(u^{\epsilon,\delta}(z,r))dzdr
\end{eqnarray*}
Let
\begin{eqnarray}D_{y,s}u^{\epsilon,\delta}(x,t)=\sigma(u^{\epsilon,\delta}(y,s))S_{y,s}^{\epsilon,\delta}(x,t) \label{Du separation}
\end{eqnarray}
and then 
$S_{y,s}^{\epsilon,\delta}(x,t)$ is the solution of
\begin{eqnarray*}
S_{y,s}^{\epsilon,\delta}(x,t)&=&G_{t-s}(x,y)+\int_{s}^{t}\int_{0}^{1}G_{t-r}(x,z)\sigma^{'}(u^{\epsilon,\delta}(z,r))S_{y,s}^{\epsilon,\delta}(z,r)W(dzdr)\\
&&+\int_{s}^{t}\int_{0}^{1}G_{t-r}(x,z)[-f^{'}(u^{\epsilon,\delta}(z,r))+\frac{1}{\delta}k_{1}^{'}-\frac{1}{\epsilon}k_{2}^{'}]S_{y,s}^{\epsilon,\delta}(z,r)dyds.
\end{eqnarray*}

According to Theorem 1.2.6 (the comparison theorem of SPDE), we have the following properties:\\
(i)$S_{y,s}^{\epsilon,\delta}\geq 0$,\\
(ii)$0\leq S_{y,s}^{\epsilon,\delta}(x,t)\leq \widehat{S}_{y,s}^{\epsilon,\delta}(x,t)$ and $\widehat{S}_{y,s}^{\epsilon,\delta}(x,t)$ is the solution of SPDE:
\begin{eqnarray}\nonumber
\widehat{S}_{y,s}^{\epsilon,\delta}&=&G_{t-s}(x,y)+\int_{s}^{t}\int_{0}^{1}G_{t-r}(x,z)\sigma^{'}(u^{\epsilon,\delta}(z,r))\widehat{S}_{y,s}^{\epsilon,\delta}(z,r)W(dzdr)\\
&&+\int_{s}^{t}\int_{0}^{1}G_{t-r}(x,z)[-f^{'}(u^{\epsilon,\delta}(z,r))]\widehat{S}_{y,s}^{\epsilon,\delta}(z,r)dzdr. \label{S equation}
\end{eqnarray}
Consequently,
\begin{equation}
|D_{y,s}u^{\epsilon,\delta}(x,t)|= |\sigma(u^{\epsilon,\delta}(y,s))|S_{y,s}^{\epsilon,\delta}(x,t)\leq |\sigma(u^{\epsilon,\delta}(y,s))|\widehat{S}_{y,s}^{\epsilon,\delta}(x,t).
\end{equation}
According to Proposition 2.1 in \cite{YZ2}, we already have the following:
\begin{equation}
\sup_{\epsilon,\delta}E[\sup_{(y,s)\in[0,1]\times[0,T]}|u^{\epsilon,\delta}(y,s)|^{p}]<\infty.
\end{equation}
We just need to prove
\begin{equation}
\sup_{\epsilon,\delta}E(\int_{0}^{t}\int_{0}^{1}|\widehat{S}_{y,s}^{\epsilon,\delta}|^{2}dyds)^{p}< \infty,
 \forall p\geq 1,
\end{equation}
according to Theorem 1.2.2 (Lemma 1.2.3 in \cite{N}).\\
We know from (\ref{S equation}):
\begin{eqnarray*}
&&|\widehat{S}_{y,s}^{\epsilon,\delta}(x,t)|^{2}\\
&\leq &
c \{|G_{t-s}(x,y)|^{2}+|\int_{s}^{t}\int_{0}^{1}G_{t-r}(z,r)\sigma^{'}(u^{\epsilon,\delta}(z,r))\widehat{S}_{y,s}^{\epsilon,\delta}(z,r)W(dzdr)|^{2}\\
&&+|\int_{s}^{t}\int_{0}^{1}G_{t-r}(x,z)[-f^{'}(u^{\epsilon,\delta}(z,r))]\widehat{S}_{y,s}^{\epsilon,\delta}(z,r)dzdr|^{2}\}.
\end{eqnarray*}
Then,
\begin{eqnarray*}
&&|\int_{0}^{t}\int_{0}^{1}|\widehat{S}_{y,s}^{\epsilon,\delta}(x,t)|^{2}dyds|^{p} \\
&\leq& c_{p} \{(\int_{0}^{t}\int_{0}^{1}|G_{t-s}(x,y)|^{2}dyds)^{p}\\
&&+(\int_{0}^{t}\int_{0}^{1}|\int_{s}^{t}\int_{0}^{1}G_{t-r}(x,z)\sigma^{'}(u^{\epsilon,\delta}(z,r))\widehat{S}_{y,s}^{\epsilon,\delta}(z,r)W(dzdr)|^{2}dyds)^{p}\\
&&+(\int_{0}^{t}\int_{0}^{1}|\int_{s}^{t}\int_{0}^{1}G_{t-r}(x,z)[-f^{'}(u^{\epsilon,\delta}(z,r))]\widehat{S}_{y,s}^{\epsilon,\delta}(z,r)dzdr|^{2}dyds)^{p}\}.
\end{eqnarray*}

We shall use Burkholder's inequality for Hilbert space (see \cite{BP1} Inequality(4.18) P41) to get the following:
\begin{eqnarray*}
&&E|\int_{0}^{t}\int_{0}^{1}|\widehat{S}_{y,s}^{\epsilon,\delta}(x,t)|^{2}dyds|^{p}\\
&\leq &c_{p} \{M\\
&&+E(\int_{0}^{t}\int_{0}^{1}|\int_{s}^{t}\int_{0}^{1}G_{t-r}(x,z)\sigma^{'}(u^{\epsilon,\delta}(z,r))\widehat{S}_{y,s}^{\epsilon,\delta}(z,r)W(dzdr)|^{2}dyds)^{p}\\
&&+E(\int_{0}^{t}\int_{0}^{1}|\int_{s}^{t}\int_{0}^{1}G_{t-r}(x,z)[-f^{'}(u^{\epsilon,\delta}(z,r))]\widehat{S}_{y,s}^{\epsilon,\delta}(z,r)dzdr|^{2}dyds)^{p}\}\\
&\leq &c_{p} \{M\\
&&+KE(\int_{0}^{t}\int_{0}^{1}(\int_{0}^{r}\int_{0}^{1}G_{t-r}^{2}(x,z)(\sigma^{'}(u^{\epsilon,\delta}(z,r)))^{2}(\widehat{S}_{y,s}^{\epsilon,\delta}(z,r))^{2}dyds)dzdr)^{p}\\
&&+E(\int_{0}^{t}\int_{0}^{1}|\int_{s}^{t}\int_{0}^{1}G_{t-r}^{2}(x,z)[-f^{'}(u^{\epsilon,\delta}(z,r))]^{2}(\widehat{S}_{y,s}^{\epsilon,\delta}(z,r))^{2}dzdr|dyds)^{p}\}\\
&\leq& c_{p}\{M+KE|\int_{0}^{t}\int_{0}^{1}(\int_{0}^{r}\int_{0}^{1}G_{t-r}^{2}(x,z)(\widehat{S}_{y,s}^{\epsilon,\delta}(z,r))^{2}dyds)dzdr|^{p}\}\\
&=& c_{p} \{M+KE(\int_{0}^{t}\int_{0}^{1}G_{t-r}^{2}(x,z)[\int_{0}^{r}\int_{0}^{1}(\widehat{S}_{y,s}^{\epsilon,\delta}(z,r))^{2}dyds]dzdr)^{p}\}\\
&\leq& c_{p}M
+c_{p}KE\{(\int_{0}^{t}\int_{0}^{1}G_{t-r}^{2\epsilon q}dzdr)^{\frac{p}{q}}\cdot \int_{0}^{t}\int_{0}^{1}G_{t-r}^{2(1-\epsilon)p}[\int_{0}^{t}\int_{0}^{1}(\widehat{S}_{y,s}^{\epsilon,\delta}(z,r))^{2}dyds]^{p}dzdr\},
\end{eqnarray*}
where $\epsilon \in (1-\frac{3}{2p},\frac{3}{2}-\frac{3}{2p}), q=\frac{p}{p-1}$.\\
Then,
\begin{eqnarray*}
&&E|\int_{0}^{t}\int_{0}^{1}|\widehat{S}_{y,s}^{\epsilon,\delta}(x,t)|^{2}dyds|^{p}\\
&\leq& c_{p}M+c_{p}KM\int_{0}^{t}\int_{0}^{1}G_{t-r}^{2(1-\epsilon) p}E[\int_{0}^{r}\int_{0}^{1}(\widehat{S}_{y,s}^{\epsilon,\delta}(z,r))^{2}dyds]^{p}dzdr\\
&\leq& c_{p}M+c_{p}KM\int_{0}^{t}\sup_{z}E(\int_{0}^{r}\int_{0}^{1}(\widehat{S}_{y,s}^{\epsilon,\delta}(z,r))^{2}dyds)^{p}(\int_{0}^{1}G_{t-r}^{2(1-\epsilon)p}dz)dr\\
&\leq& c_{p}M+c_{p}KM\int_{0}^{t}\sup_{z}E[\int_{0}^{r}\int_{0}^{1}(\widehat{S}_{y,s}^{\epsilon,\delta}(z,r))^{2}dyds]^{p}(t-r)^{a}dr
\end{eqnarray*}
where $ a=\frac{1}{2}-(1-\epsilon)p.$\\
It's equivalent to
 \begin{eqnarray}
\nonumber
&&\sup_{x}E[\int_{0}^{t}\int_{0}^{1}(\widehat{S}_{y,s}^{\epsilon,\delta}(x,t))^{2}dyds]^{p}\\
&\leq& c_{p}M+c_{p}KM\int_{0}^{t}\sup_{z}E[\int_{0}^{r}\int_{0}^{1}(\widehat{S}_{y,s}^{\epsilon,\delta}(z,r))^{2}dyds]^{p}(t-r)^{a}dr
\end{eqnarray}
Let \begin{equation}
f(t)=\sup_{x}E[\int_{0}^{t}\int_{0}^{1}(\widehat{S}_{y,s}^{\epsilon,\delta}(x,t))^{2}dyds]^{p}.
\end{equation}
Then,
\begin{equation}
f(t)\leq c_{p}M+c_{p}KM\int_{0}^{t}(t-r)^{a}f(r)dr
\end{equation}
According to Gronwall's Inequality, we have,
\begin{eqnarray*}
f(t)
&\leq& c_{p}M+\int_{0}^{t}c_{p}Mc_{p}KM(t-r)^{a}exp(\int_{r}^{t}(t-s)^{a}ds)dr\\
&=& C+\int_{0}^{t}C(t-r)^{a}e^{-\frac{1}{a+1}(t-r)^{a+1}}dr\\
&=&C+C^{'}(e^{\frac{1}{a+1}t^{a+1}}-1)\\
&<&\infty.
\end{eqnarray*}
It shows that
\begin{eqnarray}
 \sup_{x}E[\int_{0}^{t}\int_{0}^{1}(\widehat{S}_{y,s}^{\epsilon,\delta}(x,t))^{2}dyds]^{p}\leq C+C^{'}(e^{\frac{1}{a+1}t^{a+1}}-1). \label{estimation}
 \end{eqnarray} 
We can deduce from (\ref{estimation}) that: 
\begin{eqnarray*}
 \sup_{\epsilon,\delta}E[\int_{0}^{t}\int_{0}^{1}(\widehat{S}_{y,s}^{\epsilon,\delta}(x,t))^{2}dyds]^{p}<\infty, \forall p\geq 1
\end{eqnarray*}
$\Box$

\begin{theorem}
If $u$ is the solution of SPDE with two walls $(u_{0};0,0;f,\sigma;h^{1},h^{2})$ and $\sigma>0$ on $[h^1,h^2]$. Then, for all $(x_{0},t_{0}) \in (0,1)\times R^{+*}$, the restriction on $(h^1(x_0,t_0),h^2(x_0,t_0))$ of the law of $u(x_{0},t_{0})$ is absolutely continuous.
\end{theorem}
we will show that, for all $a>0$, the restriction on $[h^{1}(x_{0},t_{0})+a,h^{2}(x_{0},t_{0})-b]$, the law of $u(x_{0},t_{0})$ is absolute continuous.  From Proposition 2.2 in  \cite{BH} and Proposition 3.3 in \cite{DP1}, it remains to prove if $\sigma >0$, then,
$||Du(x_{0},t_{0})||_{L^{2}([0,1]\times R^{+})}>0$ 
on 
$$\Omega_{a,b}=\{u(x_{0},t_{0})-h^{1}(x_{0},t_{0})\geq a, h^{2}(x_{0},t_{0})-u(x_{0},t_{0}) \geq b\}.$$
And,
\begin{equation}
||Du(x_{0},t_{0})||_{L^{2}(R^{+}\times[0,1])}>0
\Leftrightarrow
\int_{0}^{t_{0}}\int_{0}^{1}|D_{y,s}(u(x_{0},t_{0}))|dyds>0\  a.s.\label{positive norm}
\end{equation}
if $\sigma>0$, then $D_{y,s}u^{\epsilon,\delta}(x_{0},t_{0}) \geq 0$ by Eq(\ref{Du separation}).
By weak limit, $D_{y,s}u(x_{0},t_{0})\geq 0,$ for $(y,s)\in [0,1]\times[0,t_{0}]$.
Inequality (\ref{positive norm}) is equivalent to
\begin{equation}
\int_{0}^{t_{0}}\int_{0}^{1}D_{y,s}u(x_{0},t_{0})dyds>0 \ on \ \Omega_{a,b} \label{positive norm2}
\end{equation}
To demonstrate (\ref{positive norm2}), we will give a lower bound of $D_{y,s}u(x_{0},t_{0})$.\\
$(x_{0},t_{0}) \in (0,1)\times R^{+*}$, for $y<x_{0}$ and $s<t_{0}$,
we note $\{w(y,s;x,t);x\in[y, \widetilde{y}=(2x_{0}-y)\land 1],t>s\}$ is the solution of SPDE:
\begin{equation}
  \left\{
   \begin{aligned}
   \frac{\partial{w(x,t)}}{\partial{t}}
   -\frac{\partial^{2}w(x,t)}{\partial{x}^{2}}=\sigma^{'}(u(x,t))w(x,t)\dot{W}(x,t)+f^{'}(u(x,t))w(x,t),\\
   w(x,s)=\sigma(u(x,s)),y<x<\widetilde{y},\\
   w(y,t)=w(\widetilde{y},t)=0,t>s.\\
   \end{aligned}
  \right.  \label{linear}
  \end{equation}
 (We have omitted the dependence of $w$ of $y,s$ for abbreviation.)
 
\begin{proposition}
Suppose $a>0$ and $(x_{0},t_{0})\in (0,1)\times R^{+*} $. For $y<x_{0}$ and $s<t_{0}$,
we define $$B_{y,s}=\{w\in \Omega, \inf_{z\in [y,\widetilde{y}]}(u(z,s)-h^{1}(z,s))>\frac{a}{2} \ and\
\inf_{z\in[y,\widetilde{y}]}(h^{2}(z,s)-u(z,s))>\frac{b}{2}\},$$ $B_{y,s}$ is $\mathcal{F}_s$-measurable. If $\tau_{y,s}$ is stopping time defined by
\begin{equation}
\tau_{y,s}=\inf\{t\geq s, inf_{z\in[y,\widetilde{y}]}(u(z,t)-h^{1}(z,t))=\frac{a}{2} \ or\ inf_{z\in[y,\widetilde{y}]}(h^{2}(z,t)-u(z,t))=\frac{b}{2} \}.
\end{equation}
Then,
\begin{equation}
\int_{y}^{\widetilde{y}}D_{z,s}u(x_{0},t_{0})dz\geq w(y,s;x_{0},t_{0})I_{\{\tau_{y,s}>t_{0}\}} a.s.
\end{equation}
$w(y,s;x,t)$ is the solution of (\ref{linear}) and $w(y,s;x_{0},t_{0})>0$ a.s.
\end{proposition}
\begin{lemma}
$v^{\epsilon,\delta}(y,s;x,t)\geq w^{\epsilon,\delta}(y,s;x,t),\forall t>s, x\in [y,\widetilde{y}].$ a.s.
\end{lemma}
\begin{lemma}
There exists a subsequence of $w^{\epsilon,\delta}$ (we still note it $w^{\epsilon,\delta}$) such that
\begin{eqnarray*}
w^{\epsilon,\delta}(y,s;x_{0},t_{0}\wedge \tau_{y,s})I_{B_{y,s}} \longrightarrow w(y,s;x_{0},t_{0}\wedge \tau_{y,s})I_{B_{y,s}},
\end{eqnarray*}
and
$w(y,s;x,t)$ is solution of SPDE(\ref{linear}) which can be written as integral:
\begin{eqnarray*}
w(y,s;x,t)
&=&\int_{y}^{\widetilde{y}}\widetilde{G_{t-s}}(x,z)\sigma(u(z,s))dz\\
&+&\int_{s}^{t}\int_{y}^{\widetilde{y}}\widetilde{G_{t-r}}(x,z)\sigma^{'}(u(z,r))w(y,s;z,r)W(dzdr)\\
&+&\int_{s}^{t}\int_{y}^{\widetilde{y}}\widetilde{G_{t-r}}(x,z)f^{'}(u(z,r))w(y,s;z,r)dzdr, t>s, y<x<\widetilde{y}.
\end{eqnarray*}
\end{lemma}
We leave the proofs of Lemma 3.3.1 and 3.3.2 to the end of this section.

\textbf{Demonstration of Proposition 3.3.1:}
Observe first that
$B_{y,s}=\{\tau_{y,s}>s\}$ by continuity of u and
\begin{eqnarray*}
\{\tau_{y,s}>t_{0}\}
&=&\{w,\inf_{z\in[y,\widetilde{y}], r\in[s,t_{0}]}(u(z,r)-h^{1}(z,r))>\frac{a}{2}\ and\\
&&\inf_{z\in[y,\widetilde{y}], r\in[s,t_{0}]}(h^{2}(z,r)-u(z,r))>\frac{b}{2}\},
\end{eqnarray*}
fix $(y,s)\in [0,x_{0})\times[0,t_{0}).$
According to Proposition 3.3.2,
$\int_{y}^{\widetilde{y}}D_{z,s}u(x_{0},t_{0})dz$ is the weak limit in $L^{p}(\Omega)$ of the subsequence of $\int_{y}^{\widetilde{y}}D_{z,s}u^{\epsilon,\delta}(x_{0},t_{0})dz$.\\
Note $v(y,s;x,t):=\int_{y}^{\widetilde{y}}D_{z,s}u(x,t)dz$, and $v^{\epsilon,\delta}(y,s;x,t):=\int_{y}^{\widetilde{y}}D_{z,s}u^{\epsilon,\delta}(x,t)dz,$ for $s<t$.\\
$v^{\epsilon,\delta}$ is the solution of linear SPDE:
\begin{eqnarray*}
&&v^{\epsilon,\delta}(y,s;x,t)\\
&=&\int_{y}^{\widetilde{y}}G_{t-s}(x,z)\sigma(u^{\epsilon,\delta}(z,s))dz\\
&&+\int_{s}^{t}\int_{0}^{1}G_{t-r}(x,z)\sigma^{'}(u^{\epsilon,\delta}(z,r))v^{\epsilon,\delta}(y,s;z,r)W(dzdr)\\
&&+\int_{s}^{t}\int_{0}^{1}G_{t-r}(x,z)f^{'}_{\epsilon,\delta}(u^{\epsilon,\delta}(z,r))v^{\epsilon,\delta}(y,s;z,r)drdz, t>s;\\
&&f^{'}_{\epsilon,\delta}(u^{\epsilon,\delta}(z,r))\\
&=&[f(u^{\epsilon,\delta}(z,r))+\frac{1}{\delta}k_1-\frac{1}{\epsilon}k_2]^{'}.
\end{eqnarray*}
Introduce $w^{\epsilon,\delta}(y,s;x,t)$ to be the solution of the same SPDE as 
$v^{\epsilon,\delta}(y,s;x,t)$ restricted in the interval $[y,\widetilde{y}]$ with Dirichlet conditions at $y, \widetilde{y}$.

\begin{equation}
  \left\{
   \begin{aligned}
   \frac{\partial{w^{\epsilon,\delta}(x,t)}}{\partial{t}}
 -\frac{\partial^{2}w^{\epsilon,\delta}(x,t)}{\partial{x}^{2}}&=&\sigma^{'}(u^{\epsilon,\delta}(x,t))w^{\epsilon,\delta}(x,t)\dot{W}(x,t)\\
 &&+f^{'}_{\epsilon,\delta}(u^{\epsilon,\delta}(x,t))w^{\epsilon,\delta}(x,t);\\
   w^{\epsilon,\delta}(x,s)=\sigma(u^{\epsilon,\delta}(x,s)),y<x<\widetilde{y};\\
   w^{\epsilon,\delta}(y,t)=w^{\epsilon,\delta}(\widetilde{y},t)=0,t>s.\\
   \end{aligned}
  \right.  \label{w's sequence}
\end{equation}
(We have omitted the dependence of $w^{\epsilon,\delta}$ of $y,s$ for abbreviation.)\\
We have the integral form:
\begin{eqnarray*}
w^{\epsilon,\delta}(y,s;x,t)
&=&\int_{y}^{\widetilde{y}}\widetilde{G_{t-s}}(x,z)\sigma(u^{\epsilon,\delta}(z,s))dz\\
&+&\int_{s}^{t}\int_{y}^{\widetilde{y}}\widetilde{G_{t-r}}(x,z)\sigma^{'}(u^{\epsilon,\delta}(z,r))w^{\epsilon,\delta}(y,s;z,r)W(dzdr)\\
&+&\int_{s}^{t}\int_{y}^{\widetilde{y}}\widetilde{G_{t-r}}(x,z)f^{'}_{\epsilon,\delta}(u^{\epsilon,\delta}(z,r))w^{\epsilon,\delta}(y,s;z,r)dzdr,\\
&& t>s, y<x<\widetilde{y},
\end{eqnarray*}
where $f^{'}_{\epsilon,\delta}(u^{\epsilon,\delta}(z,r))
=[f(u^{\epsilon,\delta}(z,r))+\frac{1}{\delta}k_1-\frac{1}{\epsilon}k_2]^{'}.$\\
$\widetilde{G}$ denotes the fundamental solution of the heat equation with Dirichlet conditions on $y$ and $\widetilde{y}$($\widetilde{G}$ depends on $y$).\\
Next we will use Lemma 3.3.1 and Lemma 3.3.2 to get our result:\\
Note: $v(y,s;x_{0},t_{0})=\int_{y}^{\widetilde{y}}D_{z,s}u(x_{0},t_{0})dz\geq 0$,
\begin{eqnarray*}
v(y,s;x_{0},t_{0})&\geq& v(y,s;x_{0},t_{0})I_{\{\tau_{y,s}>t_{0}\}}\\
&=&\lim_{\epsilon,\delta \rightarrow 0}v^{\epsilon,\delta}(y,s;x_{0},t_{0})I_{\{\tau_{y,s}>t_{0}\}}\\
v^{\epsilon,\delta}(y,s;x_{0},t_{0})I_{\{\tau_{y,s}>t_{0}\}}
&\geq& w^{\epsilon,\delta}(y,s;x_{0},t_{0})I_{\{\tau_{y,s}>t_{0}\}}
\end{eqnarray*}
and
\begin{equation}
w^{\epsilon,\delta}(y,s;x_{0},t_{0})I_{\{\tau_{y,s}>t_{0}\}}\longrightarrow w(y,s;x_{0},t_{0})I_{\{\tau_{y,s}>t_{0}\}} a.s.
\end{equation}
$w(y,s;x_{0},t_{0})>0$ is a consequence of the result in Pardoux and Zhang \cite{PZ}(Proposition 3.1). $\Box$ \\

\textbf{Demonstration of Theorem 3.3.1:}
By Proposition 3.3.3, for all $s<t_{0}$ and $y<x_{0}$, there exists a measurable set $\Omega_{y,s}$ of probability 1
such that $\forall \omega \in \Omega_{y,s}$, we have:
\begin{eqnarray}
v(y,s;x_{0},t_{0})(\omega) \geq w(y,s;x_{0},t_{0})I_{\tau_{y,s}>t_{0}}(\omega)\\
and \  w(y,s;x_{0},t_{0})>0. \label{proposition's result}
\end{eqnarray}
We define $\widetilde{\Omega_{s}}=\cap_{y\in[0,x_{0})\cap Q}\Omega_{y,s}$ and then $P(\widetilde{\Omega_{s}})=1.$
In order to prove (\ref{positive norm2}), we need the following estimate.\\
By continuity of $u$, there exist two random variables $S_{0}$ and $Y_{0}$ such that $Y_{0}< x_{0}$, and $S_{0}<t_{0}$ on $\Omega_{a,b}$ and
\begin{equation}
u(z,s)-h^{1}(z,s)>\frac{a}{2}, \ h^{2}(z,s)-u(z,s)>\frac{b}{2} \  \forall r\in[S_{0},t_{0}], z\in[Y_{0},\widetilde{Y_{0}}] \ a.s.\  on\  \Omega_{a,b}
\end{equation}
A sufficient condition to prove (\ref{positive norm2}) is
\begin{equation}
\int_{S_{0}}^{t_{0}}ds\int_{0}^{1}D_{z,s}u(x_{0},t_{0})dz>0 \ on \ \Omega_{a,b} \label{positive norm3}
\end{equation}
Note $k(s)=\int_{0}^{1}D_{z,s}u(x_{0},t_{0})dz$,  (\ref{positive norm3}) can be verified if we show
$k(s)>0$ a.s. on $\Omega_{a,b}$, $\forall S_{0}\leq s\leq t_{0}$.\\
On $\Omega_{a,b}\cap \widetilde{\Omega_{s}}$,
\begin{eqnarray}
k(s)
&\geq& v(y,s;x_{0},t_{0}) \ \ \forall\  y\in Q\\
&\geq& w(y,s;x_{0},t_{0})I_{\{\tau_{y,s}>t_{0}\}}.
\end{eqnarray}
Take $y\in [Y_{0},x_{0})\cap Q$, then $$I_{\{\tau_{y,s}>t_{0}\}}=1$$ and $$k(s)\geq w(y,s;x_{0},t_{0})>0$$ according to (\ref{proposition's result}).$\Box$\\
\textbf{Demonstration of Lemma 3.3.1:}\\
The proof of Lemma 3.3.1 is the same as Proposition 5.1 and Corollary 5.1 in Appendix of \cite{DP1}.\\
\textbf{Demonstration of Lemma 3.3.2:}\\
Step 1:
we introduce the intermediate solution $\bar{w}^{\epsilon,\delta}$ of SPDE which is similar as $w^{\epsilon,\delta}$:
\begin{eqnarray*}
\bar{w}^{\epsilon,\delta}(y,s;x,t)
&=&\int_{y}^{\widetilde{y}}\widetilde{G_{t-s}}(x,z)\sigma(u^{\epsilon,\delta}(z,s))dz\\
&+&\int_{s}^{t}\int_{y}^{\widetilde{y}}\widetilde{G_{t-r}}(x,z)\sigma^{'}(u^{\epsilon,\delta}(z,r))w^{\epsilon,\delta}(y,s;z,r)W(dzdr)\\
&+&\int_{s}^{t}\int_{y}^{\widetilde{y}}\widetilde{G_{t-r}}(x,z)f^{'}(u^{\epsilon,\delta}(z,r))\bar{w}^{\epsilon,\delta}(y,s;z,r)dzdr, t>s,y<x<\widetilde{y}
\end{eqnarray*}
so that $w^{\epsilon, \delta}(y,s;x,t)-\bar{w}^{\epsilon,\delta}(y,s;x,t)$ satisfies the following PDE with random coefficients:
\begin{eqnarray} \nonumber
&&w^{\epsilon,\delta}(y,s;x,t)-\bar{w}^{\epsilon,\delta}(y,s;x,t)\\ \nonumber
&=&\int_{s}^{t}\int_{y}^{\widetilde{y}}\widetilde{G_{t-r}}(x,z)[ f^{'}_{\epsilon,\delta}(u^{\epsilon,\delta}(z,r))w^{\epsilon,\delta}(y,s;z,r)-f^{'}(u^{\epsilon,\delta}(z,r))\bar{w}^{\epsilon,\delta}(y,s;z,r)]dzdr \\
  \label{difference of w sequence}
\end{eqnarray}
Next we will show that for $t>s, x\in (y,\widetilde{y})$,
\begin{eqnarray}
[w^{\epsilon,\delta}(y,s;x,t\wedge \tau_{y,s})-\bar{w}^{\epsilon,\delta}(y,s;x,t\wedge \tau_{y,s})]I_{B_{y,s}}\longrightarrow 0. \label{limit of difference of w sequence}
\end{eqnarray}
Fix a trajectory $w\in B_{y,s}$ and consider the previous equation (\ref{difference of w sequence}) at $t\wedge \tau_{y,s}(w)$,\\
$\forall (z,r)\in [y,\widetilde{y}] \times [s,t\wedge \tau_{y,s}(w)]$, we have $u(z,r)-h^{1}(z,r)>\frac{a}{2}, h^{2}(z,r)-u(z,r)>\frac{b}{2}.$\\
Since $u^{\epsilon,\delta}$ uniformly converges to  $u$ on $ [0,T]\times [0,1],$
then there exists $\epsilon_{0}(w)>0$ such that $\epsilon<\epsilon_{0}$, $u^{\epsilon,\delta}(z,r)-h^{1}(z,r)>\frac{a}{4};$ and there exists $\delta_{0}(w)>0$ such that $\delta<\delta_{0},$ $h^{2}(z,r)-u^{\epsilon,\delta}(z,r)>\frac{b}{4}.$\\
Then for $(z,r) \in  [y,\widetilde{y}] \times [s,t\wedge \tau_{y,s}],$ we have $f^{'}_{\epsilon,\delta}(u^{\epsilon,\delta}(z,r))=f^{'}(u^{\epsilon,\delta}(z,r)),$ for $\epsilon<\epsilon_{0}, \delta <\delta_{0}$.\\
For $t>s, x\in [y,\widetilde{y}],$
\begin{eqnarray*}
&&w^{\epsilon,\delta}(y,s;x,t\wedge \tau_{y,s}(w))-\bar{w}^{\epsilon,\delta}(y,s;x,t\wedge \tau_{y,s}(w))(w)\\
&=&\int_{s}^{t\wedge \tau_{y,s}}\int_{y}^{\widetilde{y}}\widetilde{G_{t-r}}(x,z)[f^{'}(u^{\epsilon,\delta}(z,r))(w^{\epsilon,\delta}(y,s;z,r)-\bar{w}^{\epsilon,\delta}(y,s;z,r))]dzdr.
\end{eqnarray*}
Then,
\begin{eqnarray*}
&&|w^{\epsilon,\delta}(y,s;x,t\wedge \tau_{y,s}(w))-\bar{w}^{\epsilon,\delta}(y,s;x,t\wedge \tau_{y,s}(w))(w)|^{2}\\
&=&|\int_{s}^{t\wedge \tau_{y,s}}\int_{y}^{\widetilde{y}}\widetilde{G_{t-r}}(x,z)[f^{'}(u^{\epsilon,\delta}(z,r))(w^{\epsilon,\delta}(y,s;z,r)-\bar{w}^{\epsilon,\delta}(y,s;z,r))]dzdr|^{2}\\
&\leq&K\int_{s}^{t\wedge \tau_{y,s}}\int_{y}^{\widetilde{y}} \widetilde{G_{t-r}}^{2}(x,z)dzdr \int_{s}^{t\wedge \tau_{y,s}}\int_{y}^{\widetilde{y}}|w^{\epsilon,\delta}(y,s;z,r)-\bar{w}^{\epsilon,\delta}(y,s;z,r)|^{2}dzdr.
\end{eqnarray*}
We deduce that
\begin{eqnarray*}
&&\sup_{x}|w^{\epsilon,\delta}(y,s;x,t\wedge \tau_{y,s}(w))-\bar{w}^{\epsilon,\delta}(y,s;x,t\wedge \tau_{y,s}(w))(w)|^{2}\\
&\leq& KM_{t}\int_{s}^{t\wedge \tau_{y,s}}\sup_{z}|w^{\epsilon,\delta}(y,s;z,r)-\bar{w}^{\epsilon,\delta}(y,s;z,r)|^{2}(\widetilde{y}-y)dr.
\end{eqnarray*}
According to Gronwall's Lemma:
\begin{equation}
\sup_{x}|w^{\epsilon,\delta}(y,s;x,t\wedge \tau_{y,s}(w))-\bar{w}^{\epsilon,\delta}(y,s;x,t\wedge \tau_{y,s}(w))(w)|^{2}=0. a.s.
\end{equation}
Then,
\begin{equation}
|w^{\epsilon,\delta}(y,s;x,t\wedge \tau_{y,s}(\omega))(\omega)-\bar{w}^{\epsilon,\delta}(y,s;x,t\wedge \tau_{y,s}(\omega))(\omega)|=0 \ for\  \epsilon<\epsilon_{0}, \delta<\delta_{0}.   \label{convergence}
\end{equation}
We have proved (\ref{limit of difference of w sequence}).

Step 2:
$\bar{w}^{\epsilon,\delta} \longrightarrow w$\\
Note that the sequence of $w^{\epsilon,\delta}$ and $\bar{w}^{\epsilon,\delta}$ are bounded in $L^{p}(\Omega; L^{p}([y,\widetilde{y}]\times[s,t] ))$ $i.e.$
\begin{eqnarray}
\sup_{\epsilon,\delta}E[\int_{s}^{t}\int_{y}^{\widetilde{y}}(w^{\epsilon,\delta}(y,s;z,r))^{p}drdz]<\infty,
\label{w sequence}
\end{eqnarray}
\begin{eqnarray}
\sup_{\epsilon,\delta}E[\int_{s}^{t}\int_{y}^{\widetilde{y}}(\bar{w}^{\epsilon,\delta}(y,s;z,r))^{p}dzdr]<\infty,  \label{w bar sequence}
\end{eqnarray}
The convergence $a.s.$ obtained in (\ref{limit of difference of w sequence}) together with Inequalities (\ref{w sequence}) and (\ref{w bar sequence}) obtained for p, ensuring the convergence of
\begin{equation*}
[w^{\epsilon,\delta}(y,s;\cdot,\cdot \wedge \tau_{y,s})-\bar{w}^{\epsilon,\delta}(y,s;\cdot,\cdot \wedge \tau_{y,s})]I_{B_{y,s}} \ to \ 0
\end{equation*}
in $L^{p}(\Omega;L^{p}([y,\widetilde{y}]\times[s,T]))$, that is to say
\begin{equation*}
E[\int_{s}^{T\wedge \tau_{y,s}}\int_{y}^{\widetilde{y}}(w^{\epsilon,\delta}(y,s;z,r)-\bar{w}^{\epsilon,\delta}(y,s;z,r))^{p}dzdr]\longrightarrow 0, \ \ \epsilon,\delta \rightarrow 0
\end{equation*}
\begin{eqnarray*}
&&w(x,t)-\bar{w}^{\epsilon,\delta}(x,t)\\
&=&\int_{y}^{\widetilde{y}}\widetilde{G_{t-s}}(x,z)[\sigma(u(z,s))-\sigma(u^{\epsilon,\delta}(z,s))]dz\\
&&+\int_{s}^{t}\int_{y}^{\widetilde{y}}\widetilde{G_{t-r}}(x,z)[\sigma^{'}(u(z,r))w(z,r)-\sigma^{'}(u^{\epsilon,\delta}(z,r))w^{\epsilon,\delta}(y,s;z,r)]W(dzdr)\\
&&+\int_{s}^{t}\int_{y}^{\widetilde{y}}\widetilde{G_{t-r}}(x,z)[f^{'}(u(z,r))w(z,r)-f^{'}(u^{\epsilon,\delta}(z,r))\bar{w}^{\epsilon,\delta}(y,s;z,r)]dzdr,\\
&&\ \ \ for \ \ \  t>s, y<x<\widetilde{y}
\end{eqnarray*}
Let
\begin{equation}
F^{\epsilon,\delta}(t)=\sup_{x\in[y,\widetilde{y}]}E[|w(x,t\wedge \tau_{y,s})-\bar{w}^{\epsilon,\delta}(x,t\wedge \tau_{y,s})|^{p}I_{B_{y,s}}], t>s
\end{equation}
Following the similar steps as $P.417$ in \cite{DP1}, we can show
\begin{equation}
F^{\epsilon,\delta}(t)\leq K_{p}(C^{\epsilon,\delta}+\int_{s}^{t}F^{\epsilon,\delta}(r)dr) \ and \ C^{\epsilon,\delta}\longrightarrow 0
\end{equation}
From Gronwall Lemma:
$F^{\epsilon,\delta}(t)\longrightarrow 0, \ \epsilon,\delta \rightarrow 0$\\
So we have a subsequence of $\bar{w}^{\epsilon,\delta}$ (still denote it $\bar{w}^{\epsilon,\delta}$) such that
\begin{equation}
|w(x,t\wedge \tau_{y,s})-\bar{w}^{\epsilon,\delta}(x,t\wedge \tau_{y,s})|^{p}I_{B_{y,s}}\longrightarrow 0 \  (\epsilon,\delta \rightarrow 0).
\end{equation}
$\Box$
\end{proof}
\section*{Acknowledgements}
I am grateful to Professor Tusheng Zhang for useful comments.

\end{document}